\documentclass[11pt,leqno]{amsart}
\usepackage{amsmath,amsthm,amsfonts,amssymb}
\usepackage[T1]{fontenc}
\usepackage{amsmath,amsthm, amsfonts,amssymb}
\usepackage{enumerate}
\usepackage{hyperref}
\usepackage[T1]{fontenc}
\newtheorem{thm}{Theorem}[section]

\newtheorem{cor}[thm]{Corollary}

\theoremstyle{definition}
\newtheorem{defi}[equation]{Definition}

\theoremstyle{remark}
\newtheorem*{remark}{Remark}
\newtheorem*{remark1}{Remark 1}
\newtheorem*{remark2}{Remark 2}
\newtheorem*{remark3}{Remark 3}

\newcommand{\la}{\lambda}

\newcommand{\M}{\mathcal{M}}

\DeclareMathOperator{\Dom}{Dom}

\DeclareMathOperator{\supp}{supp}

\title[Consequences of a Mihlin-H\"ormander functional calculus]{On the consequences of a Mihlin-H\"ormander functional calculus: maximal and square function estimates}
\author{B\l a\.{z}ej Wr\'obel}
\address{Dipartimento di Matematica e Applicazioni, Universit\`{a} di Milano-Bicocca,
via R. Cozzi 53 I-20125, Milano, Italy,
\newline \&
Instytut Matematyczny, Uniwersytet Wroc\l awski, pl. Grunwaldzki 2/4, 50-384 Wroc\l aw, Poland}
\email{blazej.wrobel@math.uni.wroc.pl}
 
\begin{document}

 \begin{abstract}
We prove that the existence of a Mihlin-H\"ormander functional calculus for an operator $L$ implies the boundedness on $L^p$ of both the maximal operators and the continuous square functions build on spectral multipliers of $L.$ The considered multiplier functions are finitely smooth and satisfy an integral condition at infinity. In particular multipliers of compact support are admitted.
\end{abstract}

  \maketitle
 \numberwithin{equation}{section}
 \section{Introduction}

Consider the Laplacian $-\Delta$ on $\mathbb{R}^d$ and let $m\colon [0,\infty)\to\mathbb{C}$ be a bounded compactly supported function having $[d/2]+1$ continuous derivatives. The Mihlin-H\"ormander multiplier theorem \cite{Horm1}, \cite{Mikhlin}, implies the boundedness of the Fourier multiplier operator $m(-\Delta)$  on all $L^p,$ $1<p<\infty,$ spaces. Moreover it is well known that the maximal function $M_m(f):=\sup_{t>0}|m(-t\Delta)(f)|$ has the same boundedness properties, as well as the square function $S_m(f)^2=\int_0^{\infty}|m(-t\Delta)f|^2\,\frac{dt}{t}$ (under the additional assumption $m(0)=0$).

The boundedness of $M_m$ on $L^p$ spaces, $1<p<\infty,$ follows from majorizing this operator by the Hardy-Littlewood maximal function. As for the square function $S_m,$ the boundedness on all $L^p$ spaces, $1<p<\infty,$ can be proved by applying the vector-valued Calde\'on-Zygmund theory. We remark that all the proofs referred to in this paragraph are using dilation properties of $\mathbb{R}^d$ and the Fourier transform in a decisive manner.


Replacing the Laplacian $-\Delta$ by some other non-negative self-adjoint operator $L$ on $L^2$ leads to the spectral multipliers $m(L).$ If $m\colon [0,\infty)\to \mathbb{C}$ is Borel measurable and bounded then $m(L)$ is initially well defined (and bounded) on $L^2$ by the spectral theorem. Throughout the paper we assume that $L$ has a Mihlin-H\"ormander functional calculus on $L^p,$ i.e.\ the Mihlin-H\"ormander multiplier theorem holds for $L$, regarded as an operator on $L^p$ spaces, $1<p<\infty$. More precisely, we impose that every function satisfying the Mihlin-H\"ormander condition $\sup_{\la>0}|\la ^\beta \frac{d^\beta}{d\la^\beta}m(\la)|\leq C_\beta,$ $\beta=0,\ldots,\alpha,$ of order $\alpha$ (here $\alpha=\alpha(L)$ is a fixed parameter) gives rise to a bounded operator $m(L)$ on all $L^p$ spaces, $1<p<\infty.$ There is a vast literature giving (or implying) the existence of a Mihlin-H\"ormander functional calculus for more general operators, see e.g.\ \cite{Alex_1}, \cite{s1}, \cite{s2}, \cite{s3}, \cite{s4}, \cite{s5}, \cite{s6}, \cite{s7}, \cite{JoSaTh1}, \cite{MarSik1}, \cite{MarMul1}, \cite{vectvalMM}, \cite{Meda1}, \cite{Sik}, and references therein.

On the other hand the topic of $L^p$ estimates for the general maximal operators  $M_m(f):=\sup_{t>0}|m(tL)(f)|$ and square functions $S_m(f)^2=\int_0^{\infty}|m(tL)f|^2\,\frac{dt}{t}$  has attracted considerably less attention. The main aim of this paper is to prove that, to a large degree, these estimates can be deduced from the Mihlin-H\"ormander functional calculus itself.


Let $\varphi$ be a compactly supported function with continuous $\alpha+2$ derivatives and such that $0\not \in \supp \varphi.$
One of the well known consequences of a Mihlin-H\"ormander functional calculus on $L^p,$ $1<p<\infty,$ is the $L^p$ boundedness of the discrete square function
$$S^{disc}_{\varphi}(f)=\big(\sum_{k\in \mathbb{Z}} |\varphi(2^{-k}L)f|^2\big)^{1/2}.$$
The aforementioned boundedness follows easily from Khintchine's inequality. As a consequence, for such a $\varphi,$ the discrete maximal function $M^{disc}_{\varphi}(f)=\sup_{k\in\mathbb{Z}} |\varphi(2^{-k}L)f|$ is also bounded on all $L^p$ spaces, $1<p<\infty.$ For the Fourier case we know that the assumption $0\not\in \supp\varphi$ is superfluous. We prove that the same is true for general maximal operators $M_{\varphi}(f)=\sup_{t>0}|\varphi(tL)f|$, where it might well be that $\varphi(0)\neq 0.$ This is done in Theorem \ref{thm:MHimpliesMaximal} and Corollary \ref{thm:CkimpliesMaximal}.

What concerns the square function $S_{\psi}(f)^2=\int_0^{\infty}|\psi(tL)f|^2\,\frac{dt}{t}$, the paper \cite{CowDouMcYa} by Cowling, Doust, McIntosh, and Yagi treats exhaustively the case when $\psi$ is holomorphic. Our article relaxes this assumption to some finite order of smoothness, see Theorem \ref{thm:MHimpliesSqu} and Corollary \ref{thm:CKimpliesSqu}. For instance we may take $\psi$ to be a compactly supported function having continuous $\alpha+2$ derivatives and satisfying $\psi(0)=0.$

The methods we use to treat maximal functions are based on Mellin transform techniques and have their roots in Cowling's \cite{Hanonsemi}. These techniques were employed by Alexopoulos and Lohu\'e in \cite{AlexLo} and by Gunawan and Sikora in \cite{SikGun1}. The paper \cite{AlexLo} focuses on the specific maximal functions associated with Bochner-Riesz means on general Lie groups of polynomial volume growth or on Riemannian manifolds of non-negative curvature. The report \cite{SikGun1} studies maximal functions associated with more general multipliers for elliptic operators on $\mathbb{R}^d.$ Our contribution is the observation that similar methods can be used in a far bigger generality. The techniques we employ to examine square functions are an adaptation of those from \cite{CowDouMcYa} and \cite{Me_gfun}.


It seems that the boundedness of maximal and square functions ($M_{\varphi}$ and $S_{\psi}$) for general operators $L$ having a Mihlin-H\"ormander functional calculus has not yet been fully considered. In the present paper we prove some results in this direction.
 Theorems \ref{thm:MHimpliesMaximal} and Corollary \ref{thm:CkimpliesMaximal} as well as \ref{thm:CKimpliesSqu} and Corollary \ref{thm:CKimpliesSqu} apply to all the operators considered in \cite{Alex_1}, \cite{s1}, \cite{s2}, \cite{s3}, \cite{s4}, \cite{s5}, \cite{s6}, \cite{s7}, \cite{JoSaTh1}, \cite{MarSik1}, \cite{MarMul1}, \cite{Meda1}, \cite{Sik}. In particular $L$ can be: a Laplacian on a general Lie group of polynomial volume growth (the Mihlin-H\"ormander functional calculus follows from \cite{s1}), a Laplacian on a discrete groups with polynomial volume growth (the Mihlin-H\"ormander functional calculus can be deduced from \cite{Alex_1}), or a non-negative operator having Davies-Gaffney estimates for its heat semigroup (the Mihlin-H\"ormander functional calculus is a consequence of \cite{Sik}).

The arguments used in Theorems \ref{thm:MHimpliesMaximal} and \ref{thm:MHimpliesSqu} could be easily refined to yield sharper results. We decided to stick with integral conditions $N(\varphi)<\infty$ and $\tilde{N}(\varphi)<\infty$ (see \eqref{eq:Norm}) as they can often be directly verified. Finally, let us underline that an important motivation for our research was to obtain maximal and square functions estimates for operators based on compactly supported $C^{\alpha+2}$ multipliers. This has been accomplished in Corollaries \ref{thm:CkimpliesMaximal} and \ref{thm:CKimpliesSqu}.

\section{Preliminaries}
The following terminology is used throughout the paper. The Melin transform of a function $m$ is given by
$$\M(m)(u)=\int_{0}^{\infty}s^{-iu}m(s)\frac{ds}{s},\qquad u\in \mathbb{R}.$$
The inversion formula for the Melin transform reads, for $m\in L^1((0,\infty),\frac{ds}{s})$ and $\M(m)\in L^1(\mathbb{R},du),$  as
\begin{equation}\label{eq:MelInv}m(s)=\frac1{2\pi}\int_{\mathbb{R}}\M(m)(u)s^{iu}\,du,\qquad s >0.\end{equation}
The Plancherel formula for the Melin transform is
\begin{equation}
    \label{eq:MelPlanch}
    \int_{0}^{\infty}|m(s)|^2\,\frac{ds}{s}=\frac{1}{2\pi}\int_{\mathbb{R}}|\M(m)(u)|^2\,du.
    \end{equation}

\begin{defi}
\label{def:MihHormcon}
We say that $m\colon [0,\infty)\to \mathbb{C}$ satisfies the Mihlin-H\"ormander condition of order $\alpha\in \mathbb{N}$ if $m$ is a bounded function having partial derivatives up to order $\alpha,$ and for all non-negative integers $j\leq\alpha$
\begin{equation*}\tag{MH} \label{eq:MihHorm}
\|m\|_{(\beta)}:=\sup_{\la>0}|\la^{j}\frac{d^{j}}{d\la^{j}}m(\la)|<\infty.
\end{equation*}
\end{defi}

If $m$ satisfies the Mihlin-H\"ormander condition of order $\alpha,$ then we set
\begin{equation*}
\|m\|_{MH(\alpha)}:=\sup_{\beta\leq \alpha}\|m\|_{(\beta)}.
\end{equation*}

Let $L$ be a non-negative self-adjoint operator on $L^2(X,\nu),$ for some $\sigma$-finite measure space $(X,\nu),$ with $\nu$ being a Borel measure. Then, for $m\colon [0,\infty)\to \mathbb{C},$ the spectral theorem allows us to define the multiplier operator
$m(L)=\int_{[0,\infty)}m(\la)dE(\la)$
on the domain
$$\Dom(L)=\bigg\{f\in L^2(X,\nu)\colon \int_{[0,\infty)}|m(\la)|^2\,dE_{f,f}(\la)<\infty\bigg\}.$$
Here $E$ is the spectral measure of $L,$ while $E_{f,f}$ is the complex measure given by $E_{f,f}(\cdot)=\langle E(\cdot)f,f\rangle_{L^2(X,\nu)}.$ Mainly for notational convenience throughout the paper we also assume that $L$ has trivial kernel (or in other words that the spectral projection satisfies $E(\{0\})=0$). In this case $m(L)$ can be rewritten as $m(L)=\int_{(0,\infty)}m(\la)dE(\la).$

We say that $L$ has the Mihlin-H\"ormander (MH) functional calculus of order $\alpha>0$ if the following holds: every bounded multiplier function $m$ that satisfies the Mihlin-H\"ormander condition \eqref{eq:MihHorm} of order $\alpha,$ gives rise to an operator $m(L)$ (defined initially on $L^2(X,\nu)$ by the spectral theorem), which satisfies $$\|m(L)f\|_{L^p(X,\nu)}\leq C_{p}\|m\|_{MH(\alpha)}\|f\|_{L^p(X,\nu)},\qquad f\in L^p(X,\nu)\cap L^2(X,\nu),$$ for $1<p<\infty.$ Clearly, in this case $m(L)$ extends to a bounded operator on all $L^p(X,\nu)$ spaces, $1<p<\infty.$ Throughout the paper we assume that $L$ is a non-negative self-adjoint operator that has the Mihlin-H\"ormander funcional calculus of order $\alpha$ on every $L^p(X,\nu),$ $1<p<\infty.$

Among the well-known consequences of a Mihlin-H\"ormander functional calculus the one especially important to us is the bound for the imaginary powers
\begin{equation}
\label{eq:ImaPow}
\|L^{iu}\|_{L^p(X,\nu)\to L^p(X,\nu)}\lesssim (1+|u|)^{\alpha},\qquad u\in\mathbb{R}.
\end{equation}

We say that $L$ generates a symmetric contraction semigroup whenever
 \begin{equation}
        \label{eq:contra}
        \|\exp(-tL)f\|_{L^p(X,\nu)}\leq \|f\|_{L^p(X,\nu)},\qquad f\in L^2(X,\nu)\cap L^p(X,\nu).
        \end{equation}
Note that by Meda's \cite[Theorem 4]{Meda1}, for operators generating symmetric contraction semigroups, condition \eqref{eq:ImaPow} is, up to the order of differentiability, equivalent with $L$ having a Mihlin-H\"ormander functional calculus on all $L^p$ spaces, $1<p<\infty$.

We will often consider Banach spaces $L^p(X,\nu)$ for $1<p<\infty$. For brevity we write $L^p$ and $\|\cdot\|_p$ instead of $L^p(X,\nu)$ and $\|\cdot\|_{L^p(X,\nu)},$ $1\leq p\leq\infty.$ If $T$ is a sublinear operator then the symbol $\|T\|_{p\to p}$ denotes the norm of $T$ acting on $L^p.$ Slightly abusing the terminology we say that a sublinear operator is bounded on $L^p$ if it has a unique bounded extension to $L^p.$ In particular it is enough to prove the boundedness on $L^2\cap L^p.$

The symbol $C^{\beta}(Y)$ represents the space of continuous functions on $Y$ with continuous $\beta$ derivatives. For $h\in C^{\beta}(Y)$ we define $$\|h\|_{C^{\beta}(Y)}=\sup_{\gamma\leq \beta}\sup_{y\in Y}\left|\frac{d^{\gamma}}{dy^{\gamma}}h(y)\right|;$$
note that it may happen that $\|h\|_{C^{\beta}(Y)}=\infty.$ In this paper $Y$ equals $\mathbb{R}$, $[0,\infty)$ or $[0,1];$ in the last two cases only one-sided derivatives on the boundary are considered.

For a non-negative integer $\beta$ and a Borel measurable function $\eta$ we denote
\begin{equation*}
\begin{split}
C(\eta,\beta)&:=\int_{1}^{\infty}\bigg|\frac{d^{\beta}}{ds^{\beta}}\eta(s)\bigg|s^{\beta-1}\,ds<\infty,\\
D(\eta,\beta)&:=\int_{1}^{\infty}(\log s)^2\bigg|\frac{d^{\beta}}{ds^{\beta}}\eta(s)\bigg|s^{\beta-1}\,ds<\infty
\end{split}
\end{equation*}
whenever the integrals make sense. In Theorems \ref{thm:MHimpliesMaximal} and \ref{thm:MHimpliesSqu} the following quantities are used
\begin{equation}
 \label{eq:Norm}
\begin{split}
N(\eta)&:=\|\eta\|_{C^{\alpha+2}([0,1])}+\sup_{\beta=0,1,\alpha+2}C(\eta,\beta),\\
 \tilde{N}(\eta)&:=N(\eta)+D(\eta,\alpha+2).
\end{split}
\end{equation}
Note that if $N(\eta)<\infty,$ then $\eta$ is bounded. This follows from the inequality $|\eta(s)|\leq C(\eta,1)+|\eta(1)|,$ valid for $s\geq 1.$

For non-negative numbers $A$ and $B$ by $A\lesssim_\omega B$ we mean that $A\leq C_{\omega} B,$ where $C_{\omega}$ is a constant that may depend only on $\omega$. In our case $C_{\omega}$ is always independent of the function $f\in L^2\cap L^p,$ though it may depend on both $p\in(1,\infty)$ and the multiplier $\varphi$ or $\psi$.

\section{Maximal operators}
\label{sec:MaxOp-Lp}
In this section we additionally impose that $L$ generates a symmetric contraction semigroup, i.e.\ \eqref{eq:contra} holds. This assumption holds for all the operators studied in \cite{Alex_1}, \cite{s1}, \cite{s2}, \cite{s3}, \cite{s4}, \cite{s5}, \cite{s6}, \cite{s7}, \cite{JoSaTh1}, \cite{MarSik1}, \cite{MarMul1}, \cite{vectvalMM}, \cite{Meda1}, \cite{Sik}.

The maximal functions investigated here are of the form
\begin{equation}
\label{eq:Maxdef}
M_{\varphi}(f)=\sup_{t>0}|\varphi(t L)f|.
\end{equation}
In particular, for $\varphi(\la)=\exp(-\la)$ we have the maximal operator associated with the heat semigroup $M_{\exp(-\cdot)}.$
The following is the main result of Section \ref{sec:MaxOp-Lp}. Note that formally the supremum in \eqref{eq:Maxdef} should be restricted to a countable set, however, from the proof of Theorem \ref{thm:MHimpliesMaximal} it will follow that the supremum may be taken over all of $(0,\infty).$ We recall that $N(\varphi)$ is defined by \eqref{eq:Norm}.
\begin{thm}
\label{thm:MHimpliesMaximal}
Let $\varphi\colon [0,\infty)\to \mathbb{C}$ belong to $C^{\alpha+2}([0,\infty))$ and assume that $N(\varphi)<\infty$. Then the maximal operator defined by \eqref{eq:Maxdef} is bounded on all $L^p$ spaces, $1<p<\infty,$ and
\begin{equation}
\label{eq:Maxbound}
\|M_{\varphi}(f)\|_{p}\lesssim_p N(\varphi)\|f\|_p,\qquad f\in L^2\cap L^p.
\end{equation}
\end{thm}
\begin{remark1}
 A natural question is whether it is enough to assume that $\varphi$ satisfies the Mihlin-H\"ormander condition \eqref{eq:MihHorm} instead of imposing $N(\varphi)<\infty$. In \cite{MaxFourMihHorm} the authors proved that this is not the case even for radial Fourier multipliers (which correspond to $L$ being the Laplacian). Therefore some other assumptions are indeed necessary.
\end{remark1}

\begin{remark2}
 The required smoothness threshold can be improved if we use a Sobolev (continuous) variant of the Mihlin-H\"ormander norm \eqref{eq:MihHorm} and interpolate \eqref{eq:ImaPow} with the obvious $L^2(X,\nu)$  bound $\|L^{iu}\|_{L^2(X,\nu)\to L^2(X,\nu)}\leq 1.$ However, in general the obtained result is still far from being optimal. Thus, for the clarity of the presentation we decided to phrase Theorem \ref{thm:MHimpliesMaximal} in terms of the integer Mihlin-H\"ormander norm.
\end{remark2}
\begin{remark3}
 The proof presented here can be adjusted to operators that have a Mikhlin-H\"ormander functional calculus on other Banach spaces $B$ than $L^p.$ For instance the results of \cite{DzPr1} imply the existence of such a calculus, on the Hardy space $B=H^1_L$ (introduced in \cite{HofLMiMiYa1}), for operators $L$ with Davies-Gafney estimates on their heat semigroups. However, a proper adjustment would require the strong continuity on $B$ of the imaginary powers $\{L^{iu}\}_{u\in \mathbb{R}}$, which is not clear in the case of the general Hardy space $H^1_L.$  This will be one of the problems addressed in a forthcoming paper by Jacek Dziuba\'nski and the author.
\end{remark3}

Before proceeding to the proof of Theorem \ref{thm:MHimpliesMaximal} let us note the following useful corollary.
\begin{cor}
\label{thm:CkimpliesMaximal}
If $\varphi$ is supported on a compact set $K$, then, for each $1<p<\infty,$ we have
\begin{equation}
\label{eq:MaxboundCom}
\|M_{\varphi}(f)\|_{p}\lesssim_{K,p} \|\varphi\|_{C^{\alpha+2}([0,\infty))}\|f\|_p,\qquad f\in L^2\cap L^p.
\end{equation}
\end{cor}
\begin{proof}
 Use Theorem \ref{thm:MHimpliesMaximal} and the estimate $N(\varphi)\lesssim_K \|\varphi\|_{C^{\alpha+2}([0,\infty))}.$
\end{proof}

\begin{remark}
\label{p:Mphidelta}
Exemplary functions satisfying the assumptions of the corollary are $\varphi^{\delta}=(1-\la)^{\delta}\chi_{\{\la<1\}},$ with $\delta\geq \alpha+2.$ In this case $M_{\varphi^{\delta}}$ is the maximal function connected with the Bochner-Riesz means.
\end{remark}
Now we pass to the proof of the main result of this section.
\begin{proof}[Proof of Theorem \ref{thm:MHimpliesMaximal}]
We will use Melin transform techniques here, to estimate the maximal operator by a certain integral of imaginary powers $L^{iu}$. This idea is due to Cowling \cite[Section 3]{Hanonsemi}. Remark that essentially all that is needed to write the proof below rigorously is the strong $L^p$ continuity (for $1<p<\infty$) of the group of imaginary powers.

From \cite[Theorem 7]{Hanonsemi} it follows that $M_{\exp(-\cdot)}$ is bounded on all $L^p.$ Thus, it is enough to show that $M_{\varphi(\cdot)-\varphi(0)\exp(-\cdot)}$ is also bounded on $L^p.$ For each fixed $\la, t>0,$ we have
\begin{align*}&\M(\varphi(\cdot t\la)-\varphi(0)\exp(-\cdot t\la))(u)=t^{iu}\la^{iu}\int_{0}^{\infty}(st\la)^{-iu}[\varphi(s t\la)-\varphi(0)\exp(-s t\la)]\frac{ds}{s}\\
&=t^{iu}\la^{iu}\int_{0}^{\infty}s^{-iu}[\varphi(s)-\varphi(0)\exp(-s )]\frac{ds}{s}:=A_{\varphi}(u)t^{iu}\la^{iu}, \qquad u\in \mathbb{R},
\end{align*}
and, consequently, by spectral theory,
$$\M(\varphi(\cdot tL)-\varphi(0)\exp(-\cdot tL))(u)=A_{\varphi}(u)t^{iu}L^{iu}.$$
Hence, using \eqref{eq:MelInv} we formally write
\begin{equation}
 \label{eq:Melinform1}
[\varphi(tL)-\varphi(0)\exp(-tL)](f)=\int_{\mathbb{R}}A_{\varphi}(u)t^{iu}L^{iu}f\,du,
\end{equation}
which together with Minkowski's integral inequality imply
\begin{equation*}
\|M_{\varphi(\cdot)-\varphi(0)\exp(-\cdot)}(f)\|_p=\|\sup_{t>0}\left|\int_{\mathbb{R}}A_{\varphi}(u)t^{iu}L^{iu}f\,du\right|\|_p\leq \int_{\mathbb{R}}|A_{\varphi}(u)|\|L^{iu}f\|_p\,du.
\end{equation*}
We remark that the validity of \eqref{eq:Melinform1} was exactly the reason for subtracting the term $\varphi(0)\exp(-s).$ Now, in view of \eqref{eq:ImaPow} it is enough to prove the bound
\begin{equation}
\label{eq:Cuest}
|A_{\varphi}(u)|\lesssim N(\varphi) (1+|u|)^{-\alpha-2},\qquad u\in \mathbb{R}.
\end{equation}

Assume first that $\varphi$ is compactly supported. Then, clearly, $|A_{\varphi}(u)|\lesssim N(\varphi).$ Moreover, repeated integration by parts produces
$$A_{\varphi}(u)=\frac1{-iu\cdots(-iu+\alpha+2)}\int_{0}^{\infty}s^{-iu+\alpha+1}\frac{d^{\alpha+2}}{ds^{\alpha+2}}[\varphi(s)-\varphi(0)\exp(-s )]\,ds.$$ Now a short calculation leads to \eqref{eq:Cuest}.

To demonstrate \eqref{eq:Cuest} for a general multiplier $\varphi$ we use an approximation argument. Let $\rho\colon \mathbb{R}\to \mathbb{R}$ be smooth and such that $\rho(y)=1,$ for $y\leq 0,$ and $\rho(y)=0,$ for $y>1.$ Then, for each $R>0$ the function $\rho_R(\la)=\rho(\la-R)$ is compactly supported and smooth when restricted to $\la \in [0,\infty).$ Thus, for $\varphi_R=\varphi \times \rho_R$ we have the bound
\begin{equation}
\label{eq:CuestR}
|A_{\varphi_R}(u)|\lesssim N(\varphi_R) (1+|u|)^{-\alpha-2},\qquad u\in \mathbb{R}.
\end{equation}
Now, using the assumptions on $\varphi$ and Lebesgue's dominated convergence theorem, it is straightforward to see that $\lim_{R\to \infty}A_{\varphi_R}(u)=A_{\varphi}(u)$ and  $\lim_{R\to \infty} N(\varphi_R)=N(\varphi).$ Hence, taking $R\to \infty$ in \eqref{eq:CuestR} gives \eqref{eq:Cuest} for general $\varphi.$

In summary, the proof is completed, provided we justify that the formal expression \eqref{eq:Melinform1} converges as an $L^p$-valued integral. This follows from the well-known $L^p$ continuity of $u\mapsto L^{iu}f$ and the finiteness of $\int_{\mathbb{R}}|A_{\varphi}(u)|\|L^{iu}f\|_p\,du.$ Observe also that the above argument together with the Lebesgue's dominated convergence theorem shows that the mapping $t\mapsto \int_{\mathbb{R}}A_{\varphi}(u)t^{iu}L^{iu}f(x)\,du$ is continuous, for a.e.\ $x\in X.$ Hence, in view of \eqref{eq:Melinform1} the supremum in the definition of the maximal function may be taken over all $(0,\infty).$
\end{proof}


\section{Continuous square functions}
\label{sec:SqFun-Lp}
Now we pass to square functions. These are given by
\begin{equation}
\label{eq:Sqdef}
S_{\psi}(f)=\bigg(\int_0^{\infty}|\psi(tL)f|^2\,\frac{dt}t\bigg)^{1/2}.
\end{equation}
The spectral theorem implies that $S_{\psi}$ is bounded on $L^2$ (in which case it is an isometry) if and only if \begin{equation}\label{eq:SqL2}\int_0^{\infty}|\psi(t)|^2\,\frac{dt}t<\infty.\end{equation}
The next theorem establishes the boundedness of $S_{\psi}$ on other $L^p$ spaces. Recall that $\tilde{N}(\psi)$ is defined by \eqref{eq:Norm}.
\begin{thm}
\label{thm:MHimpliesSqu}
Let $\psi\in C^{\alpha+2}[0,\infty)$ be such that $\psi(0)=0$ and $\tilde{N}(\psi)<\infty.$
Then, for each $1<p<\infty,$  the square function \eqref{eq:Sqdef} satisfies
\begin{equation}
\label{eq:MHimpliesSqu}
\|f\|_p\lesssim_{p,\psi}\|S_{\psi}(f)\|_p\lesssim_{p,\psi}\|f\|_p,\qquad f\in L^2 \cap L^p.
\end{equation}
\end{thm}
Before proving the theorem we state a seemingly interesting corollary.
\begin{cor}
\label{thm:CKimpliesSqu}
Let $\psi\in C^{\alpha+2}[0,\infty)$ be such that $\psi(0)=0$ and assume that $\psi$ is supported on a compact set $K.$
Then, for each $1<p<\infty,$ we have
\begin{equation}
\label{eq:MHimpliesSqu}
\|f\|_{p}\lesssim_{p,\psi,K}\|S_{\psi}(f)\|_p\lesssim_{p,\psi,K}\|f\|_p,\qquad f\in L^2\cap L^p.
\end{equation}
\end{cor}
\begin{proof}
 Simply observe that $\tilde{N}(\psi)<\infty.$
\end{proof}
\begin{remark}
Examples of functions admitted by the corollary are $\psi^{\delta}(\la)=\la (1-\la)^{\delta}\chi_{\{\la<1\}},$ with $\delta\geq \alpha+2.$ The usefulness of $S_{\psi^{\delta}}$ lies in the fact that sharp $L^p,$ $p>2,$ bounds for this square function imply sharp $L^p$ bounds for maximal functions $M_{\varphi^{\delta}}$ (defined on p.\ \pageref{p:Mphidelta}) associated with Bochner-Riesz means, see \cite[p.\ 274]{LeeRogSeeg}, \cite{Car1}, \cite{Car2}, and \cite{DapTre1}. This is true not only in the Fourier case but also in general as the reasoning from \cite[p.\ 54]{Car2} can be repeated mutatis mutandis.
\end{remark}
We proceed with the proof of the main result of this section.
\begin{proof}[Proof of Theorem \ref{thm:MHimpliesSqu}]
Since $\tilde{N}(\psi)<\infty$  implies the boundedness of $\psi,$ we see that \eqref{eq:SqL2} holds. Therefore the square function $S_{\psi}$ is bounded on $L^2.$

We use Melin transform techniques here together with ideas from \cite{CowDouMcYa}. By a polarization argument it is enough to prove the right hand side inequality in \eqref{eq:MHimpliesSqu}. Note that, since $\psi(0)=0,$ the function $\M(\psi)(u)$ is well defined. Moreover, for fixed $\la>0$ it holds $\M(\psi(\cdot \la))(u)=\M(\psi)(u)\la^{iu}.$ Thus using the spectral theorem and Plancherel's formula \eqref{eq:MelPlanch} the square function given by \eqref{eq:Sqdef} can be reexpressed as
\begin{equation}
\label{eq:SqdefMel}
S_{\psi}(f)=\bigg(\int_{\mathbb{R}}|\M(\psi)(u)L^{iu}f|^2\,du\bigg)^{1/2}.
\end{equation}

We start by following the scheme from \cite{Me_gfun}. Take a smooth function $h$ on $\mathbb{R},$ supported in $[-\pi,\pi],$ and such that
\begin{equation} \label{eq:resid}\sum_{k\in\mathbb{Z}} h_k(u):=\sum_{k\in\mathbb{Z}} h(u-\pi k)=1,\qquad u\in\mathbb{R}.\end{equation} For each $j\in\mathbb{Z}$ we define the functions $b_{j,k}$ on $[0,\infty)$
by
\begin{equation*}
b_{j,k}(\la)=\int_{\mathbb{R}}h_{k}(u)\M(\psi)(u)e^{-ij u} \la^{iu}\,du.
\end{equation*}
It is not hard to see that the functions $b_{j,k}$ are bounded on $[0,\infty),$ thus it makes sense to consider $b_{j,k}(L).$

From Parseval's formula for the Fourier series it follows that, for each $k\in\mathbb{Z},$ the set $\{(2\pi)^{-1/2}e^{-ij( u-\pi k) }\}_{j\in\mathbb{Z}}$ forms an orthonormal basis in the space $L^2(Q_k,du),$ with $Q_k=[k\pi-\pi,k\pi+\pi].$ Therefore, since $\supp(h_k)\subset Q_k$ and $|e^{i \pi \langle j, k\rangle}|=1,$ for each fixed $\la >0$ we have
\begin{equation}
\label{eq:Pars}
\begin{split}
&\int_{\mathbb{R}}\left|h_k(u)\M(\psi)(u)\la^{iu}\right|^2\,du\\
&=\frac{1}{(2\pi)^{1/2}}\sum_{j\in\mathbb{Z}} \left|\int_{\mathbb{R}}h_k(u)\M(\psi)(u)e^{-iju}\la^{iu}\,du\right|^2.
\end{split}
\end{equation}

Now, using \eqref{eq:SqdefMel}, \eqref{eq:resid} and Minkowski's integral inequality, followed by the spectral theorem and \eqref{eq:Pars}, we obtain
\begin{align*}
S_{\psi}(f)(x)&\leq (2\pi)^{-1/2} \sum_{k\in\mathbb{Z}^d}\left[\int_{\mathbb{R}}\left|h_k(u)\M(\psi)(u)L^{iu}f(x)\right|^2\,du\right]^{1/2}\\
&=(2\pi)^{-1}\sum_{k\in\mathbb{Z}}\left[\sum_{j\in\mathbb{Z}}\left|\int_{\mathbb{R}}h_k(u)\M(\psi)(u)e^{-ij u}L^{iu}f(x)\,du\right|^2\right]^{1/2}\\
&=(2\pi)^{-1}\sum_{k\in\mathbb{Z}}\bigg(\sum_{j\in\mathbb{Z}}|b_{j,k}(L)(f)(x)|^2\bigg)^{1/2},\qquad x-\textrm{a.e.},
\end{align*}
and consequently,
\begin{equation*}
\|S_{\psi}(f)\|_p\leq (2\pi)^{-1}\sum_{k\in\mathbb{Z}}\bigg\|\bigg(\sum_{j\in\mathbb{Z}}|b_{j,k}(L)f|^2\bigg)^{1/2}\bigg\|_p.
\end{equation*}
Then from Khintchine's inequality it follows that
\begin{equation}
\label{eq:ineqaim}
(2\pi)\|S_{\psi}(f)\|_p\leq \sum_{k\in\mathbb{Z}}\sup_{|a_{j}|\leq 1}\bigg\|\bigg(\sum_{j\in\mathbb{Z}}a_{j} b_{j,k}(L)\bigg)f\bigg\|_p
:=\sum_{k\in\mathbb{Z}}\sup_{|a_{j}|\leq 1}\bigg\|m^{a_j}_{k}(L)f\bigg\|_p.
\end{equation}

We shall now focus on estimating the $L^p$ operator norm of each of the operators $m^{a}_{k}(L),$ defined in \eqref{eq:ineqaim}.
We claim that each of the functions $$m(\la):=m^{a}_{k}(\la)=\sum_{j\in\mathbb{Z}}a_{j} b_{j,k}(\la)$$ satisfies the Mihlin-H\"ormander condition \eqref{eq:MihHorm}, uniformly in $a,$ and with a quadratic decay in $k,$ i.e.\
\begin{equation}
\label{eq:claim}
\|m\|_{MH(\alpha)}\lesssim \frac{1}{1+k^2},\qquad k\in \mathbb{Z}.
\end{equation}

If the claim is true, then, using the assumption that the operator $L$ has a Mihlin-H\"ormander functional calculus of order $\alpha,$ and coming back to \eqref{eq:ineqaim} we finish the proof of Theorem \ref{thm:MHimpliesSqu}.

Thus we focus on proving \eqref{eq:claim}. Let $n(y)=m(e^y),$ $y\in \mathbb{R}.$ Then, using e.g.\ Fa\'a di Bruno's formula, it is not hard to see that
$$\|m\|_{MH(\alpha)}\approx \|n\|_{C^{\alpha}(\mathbb{R})}.$$ Consequently, denoting
$$C_{j,k}(y):=b_{j,k}(e^y)=\int_{\mathbb{R}}h_k(u)\M(\psi)(u)e^{i(y-j) u}\,du,\qquad y\in\mathbb{R},$$
it is enough to show that
\begin{equation}
\label{eq:expreduct}
\|n\|_{C^{\alpha}(\mathbb{R})}=\bigg\|\sum_j a_j C_{j,k}\bigg\|_{C^{\alpha}(\mathbb{R})}\lesssim \frac1{1+k^2}.
\end{equation}

In order to prove \eqref{eq:expreduct} we will need the bound
\begin{equation}
\label{eq:SquMest}
\big|\M(\psi)(u)\big|+\bigg|\frac{d}{du}\M(\psi)(u)\bigg|+\bigg|\frac{d^2}{du^2}\M(\psi)(u)\bigg|\lesssim \tilde{N}(\psi)(1+|u|)^{-\alpha-2}\lesssim (1+|u|)^{-\alpha-2}.
\end{equation}
It suffices to show \eqref{eq:SquMest} for compactly supported $\psi$ as the general case follows from an approximation argument similar to the one performed in the proof of Theorem \ref{thm:MHimpliesMaximal}.

For small $|u|<1$ the finiteness of $\tilde{N}(\psi)$ implies
\begin{equation}
\label{eq:SquMestu<1}
\big|\M(\psi)(u)\big|+\bigg|\frac{d}{du}\M(\psi)(u)\bigg|+\bigg|\frac{d^2}{du^2}\M(\psi)(u)\bigg|\lesssim \tilde{N}(\psi).
\end{equation}
For large $|u|>1$ integration by parts (which is legitimate since $\psi$ is compactly supported) produces
\begin{equation*}
\M(\psi)(u)=\frac{1}{-iu\cdots(-iu+\alpha+2)}\int_{0}^{\infty}s^{-iu+\alpha+1}\frac{d^{\alpha+2}}{ds^{\alpha+2}}\psi(s)\,ds,
\end{equation*}
and, consequently,
\begin{align*}
\frac{d}{du}\M(\psi)(u)&=\frac{1}{-iu\cdots(-iu+\alpha+2)}\int_{0}^{\infty}(-i\log s)s^{-iu+\alpha+1}\frac{d^{\alpha+2}}{ds^{\alpha+2}}\psi(s)\,ds,\\
\frac{d^2}{du^2}\M(\psi)(u)&=\frac{1}{-iu\cdots(-iu+\alpha+2)}\int_{0}^{\infty}(-i\log s)^{2}s^{-iu+\alpha+1}\frac{d^{\alpha+2}}{ds^{\alpha+2}}\psi(s)\,ds.
\end{align*}
Thus, using the assumptions on $\psi$ we obtain \eqref{eq:SquMest} also for $|u|>1.$

Coming back to the proof of \eqref{eq:expreduct} we observe that for each fixed $\beta\leq \alpha$ it holds
\begin{equation}\label{eq:dxCjkform}\frac{d^{\beta}}{dy^{\beta}}C_{j,k}(y)=i^{\beta}\int_{\mathbb{R}}h_k(u)\M(\psi)(u)u^\beta e^{i(y-j) u}\,du.\end{equation}
Hence, \eqref{eq:SquMest} together with $\supp h_k\subseteq [k\pi-\pi,k\pi+\pi]$ give
\begin{equation}
\label{eq:SquCjkest1}
\bigg|\frac{d^{\beta}}{dy^{\beta}}C_{j,k}(y)\bigg|\lesssim \int_{[k\pi-\pi,k\pi+\pi]}\frac{|u|^{\alpha}}{(1+|u|)^{\alpha+2}}\,du\lesssim \frac1{1+k^2}.
\end{equation}
Using additionally integration by parts in \eqref{eq:dxCjkform} and applying \eqref{eq:SquMest} we obtain
\begin{equation}
\label{eq:SquCjkest2}
\bigg|\frac{d^{\beta}}{dy^{\beta}}C_{j,k}(y)\bigg|=\frac{1}{(y-j)^2}\bigg|\int_{\mathbb{R}}\frac{d^2}{du^2}[h_k(u)\M(\psi)(u)u^\beta] e^{i(y-j) u}\,du\bigg|\lesssim \frac{1}{1+k^2}\frac{1}{(y-j)^2}.
\end{equation}
Finally, combining \eqref{eq:SquCjkest1} and \eqref{eq:SquCjkest2} we arrive at the bound
$$\bigg|\frac{d^{\beta}}{dy^{\beta}}n(y)\bigg|\lesssim \frac1{1+k^2}\bigg(1+\sum_{j\colon|j-y|> 1}\frac{1}{(j-y)^2}\bigg)\lesssim \frac1{1+k^2},\qquad y\in\mathbb{R},$$ valid for $\beta\leq \alpha.$ The proof of \eqref{eq:expreduct} and thus also of the theorem is completed.

\end{proof}

\subsection*{Acknowledgments}
The research was supported by Polish funds for sciences, National Science Centre (NCN), Poland, Research Project 2014\slash 15\slash D\slash ST1\slash 00405.


\begin{thebibliography}{16}
\bibitem{Alex_1} G.\ Alexopoulos, \textit{Spectral multipliers on discrete groups}, Bull.\ Lond.\ Math.\ Soc.\ (4) 33 (2001), 417--424.
\bibitem{s1} G.\ Alexopoulos, \textit{Spectral multipliers on Lie groups of polynomial growth}, Proc.\ Amer.\ Math.\ Soc.\, (3) 120 (1994), 973--979.
\bibitem{s2} G.\ Alexopoulos, \textit{Spectral multipliers for Markov chains}, J.\ Math.\ Soc.\ Japan (3) 56 (2004), 833--852.
\bibitem{AlexLo} G.\ Alexopoulos, N.\ Lohou\'e, \textit{Riesz means on Lie groups and Riemannian manifolds of nonnegative curvature}, Bull.\ Soc.\ Math.\ France, 122 (1994), pp.\ 209--223.

\bibitem{Car1} A.\ Carbery, \textit{The boundedness of the maximal Bochner-Riesz operator on $L^4(\mathbb{R}^2)$}, Duke Math.\ J.\
50 (1983), pp.\ 409--416.
\bibitem{Car2} A.\ Carbery, \textit{Radial Fourier multipliers and associated maximal functions}, Recent progress in Fourier analysis (El Escorial, 1983), pp.\ 49--56, North-Holland Math.\ Stud., 111, North-Holland, Amsterdam, 1985.

\bibitem{MaxFourMihHorm} M.\ Christ, L.\ Grafakos, P.\ Honz{\'{\i}}k, and A.\ Seeger, \textit{Maximal functions associated with Fourier multipliers of Mikhlin-Hörmander type}, Math.\ Z.\, 249 (1) 2005, pp.\ 223--240.
\bibitem{s3} M.\ Christ, C.\ D.\ Sogge, \textit{The weak type $L^1$ convergence of eigenfunction expansions for pseudodifferential operators}, Invent.\ Math.\ (2) 94 (1998), 421--453.
\bibitem{s4} M.\ G.\ Cowling, A.\ Sikora, \textit{A spectral multiplier theorem on $SU(2)$}, Math.\ Z.\ (1) 238 (2001), 1--36.
\bibitem{Hanonsemi} M.\ G.\ Cowling, \textit{Harmonic analysis on semigroups}, Ann.\ of Math.\ 117 (1983), 267--283.
\bibitem{CowDouMcYa} M.\ G.\ Cowling, I.\ Doust, A.\ McIntosh, and A.\ Yagi, \textit{Banach space operators with a bounded $H^{\infty}$ functional
calculus}, Journal of Aust.\ Math.\ Society (Series A) 60 (1996), 51--89.

\bibitem{DapTre1}  H.\ Dappa, W.\ Trebels, \textit{On maximal functions generated by Fourier multipliers}, Ark.\ Mat.\ (2) 23 (1985), 241--259.
\bibitem{s5} X.\ T.\ Duong, A.\ McIntosh, \textit{Singular integral operators with non-smooth kernels on irregular domains}, Rev.\ Mat.\ Iberoamericana (2) 15 (1999), 233--265.
\bibitem{s6} X.\ T.\ Duong, E.\ M.\ Ouhabaz, and A.\ Sikora, \textit{Plancherel type estimates and sharp spectral multipliers}, J.\ Funct.\ Anal.\ (2) 196 (2002), 443--485.
\bibitem{DzPr1} J.\ Dziuba\'nski, M.\ Preisner, \textit{ Remarks on spectral multiplier theorems on Hardy spaces associared with semigroups of operators}, Rev.\ Un.\ Mat.\ Argentina,


\bibitem{EnNa1} K-J.\ Engel, R.\ Nagel, \textit{A Short Course on Operator Semigroups}, Universitext, Springer, 2006.

\bibitem{SikGun1} H.\ Gunawan, A.\ Sikora, \textit{On maximal operators associated to Laplace operators}, Research Report, http://personal.fmipa.itb.ac.id/hgunawan/files/2007/11/wmax8r.pdf , 2002.

\bibitem{s7} W.\ Hebisch, \textit{A multiplier theorem for Schr\"odinger operators}, Colloq.\ Math.\ (2) 60/61 (1990), 659--664.
\bibitem{HofLMiMiYa1}  S.\ Hofmann, G.\ Z.\ Lu, D.\ Mitrea, M.\ Mitrea, and L.\ X.\ Yan, \textit{Hardy spaces associated with
non-negative self-adjoint operators satisfying Davies-Gafney estimates}, Mem.\ Amer.\ Math.\ Soc.\ Vol.\ 214, Nr 1007, 2011.
\bibitem{Horm1} L.\ H\"{o}rmander, \textit{Estimates for translation invariant operators in $L^p$ spaces}, Acta Math.\ (1) 104 (1960), 93--140.

\bibitem{JoSaTh1} K.\ Jotsaroop, P.\ K. Sanjay, S.\ Thangavelu \textit{Riesz transforms and multipliers for the Grushin operator}, J.\ Anal.\ Math.\ (1) 119 (2013), 255--273.

\bibitem{LeeRogSeeg} S.\ Lee, K.\ M.\ Rogers, and A.\ Seeger,
\textit{Square functions and maximal operators associated with radial Fourier multipliers}, Advances in Analysis: The
Legacy of Elias M.\ Stein, Princeton Mathematical Series, Princeton University Press, 2014, pp.\ 273--302.

\bibitem{MarSik1} A.\ Martini, A.\ Sikora, \textit{Weighted Plancherel estimates and sharp spectral multipliers for the Grushin operators}, Math.\ Res.\ Lett.\ (5) 19 2012, 1075---1088.
\bibitem{MarMul1} A.\ Martini, D.\ M\"uller, \textit{A sharp multiplier theorem for Grushin operators in arbitrary dimensions}, Rev.\ Mat.\ Iberoam.\ (4) 30 (2014), 1265--1280.
\bibitem{vectvalMM} G.\ Mauceri, S.\ Meda,  \textit{Vector-valued multipliers on stratified groups}, Rev.\ Mat.\ Iberoamericana (3-4) 6 (1990), 141--154.
\bibitem{Meda1} S.\ Meda, \textit{A general multiplier theorem}, Proc.\ Amer.\ Math.\ Soc.\ (3) 110 (1990), 639--647.
\bibitem{Me_gfun} S.\ Meda, \textit{On the Littlewood-Paley-Stein g-Function}, Trans.\ Amer.\ Math.\ Soc.\ (6) 347 (1995), 2201--2212.
\bibitem{Mikhlin} S.\ G.\ Mikhlin, \textit{Multidimensional singular integrals and integral equations}, Translated from
the Russian by W.\ J.\ A.\ Whyte. Pergamon Press, Oxford-New York-Paris 1965.

\bibitem{Sik} A.\ Sikora, \textit{Multivariable spectral multipliers and analysis of quasielliptic operators on fractals}, Indiana Univ.\ Math.\ J.\ 58 (2009), 317--334.

\end{thebibliography}
\end{document}